\documentclass{article}
\usepackage{amssymb}
\usepackage{amsmath}
\usepackage{amsthm}
\usepackage{bm}
\usepackage{pb-diagram}
\usepackage{subcaption}
\newcommand{\g}{\mathfrak{g}}

\newcommand{\h}{\mathfrak{h}}
\newcommand{\mv}{\mathfrak{v}}
\newcommand{\mw}{\mathfrak{w}}
\newtheorem{definition}{Definition}[section]
\newtheorem{remark}{Remark}[section]
\newtheorem{theorem}{Theorem}[section]
\newtheorem{proposition}{Proposition}[section]

\newtheorem{lemma}{Lemma}[section]
\newtheorem{example}{Example}[section]
\newtheorem{corollary}{Corollary}[section]

\usepackage{amscd}
\newcommand{\R}{\mathbb{R}}

\newcommand{\Z}{\mathbb{Z}}
\newcommand{\N}{\mathbb{N}}
\newcommand{\T}{\mathbb{T}}
\usepackage[dvipdfmx]{graphicx}
\begin{document}

\title{Collapsed limits of compact Heisenberg manifolds with sub-Riemannian metrics}
\author{Kenshiro Tashiro}
\date{}
\maketitle

\begin{abstract}
In this paper,
we show that every collapsed Gromov--Hausdorff limit of compact Heisenberg manifolds is isometric to a flat torus.
Here we say that a sequence of sub-Riemannian manifolds collapses if their total measure with respect to the Popp's volume or the minimal Popp's volume converges to zero.

In the appendix,
we give the systolic inequality on sub-Riemannian Heisenberg manifolds,
and observe that the exponent of the total measure is equal to the inverse of the Hausdorff dimension.
\end{abstract}

\section{Introduction}

A sub-Riemannian manifold is a triple $(M,\mathcal{D},g)$,
where $M$ is a smooth manifold,
$\mathcal{D}$ is a sub-bundle of the tangent bundle,
and $g$ is a metric on $\mathcal{D}$.
In the same way to Riemannian manifolds,
we can put a length structure and the associated distace function on bracket generating sub-Riemannian manifolds (please see Definition \ref{dfnbracket}).
Sub-Riemannian manifolds appear as Gromov--Hausdorff limits of a sequence of Riemannian manifolds.
In general their sectional,
Ricci and scalar curvature diverge when the limit space is non-Riemannian.
However some sub-Riemannian manifolds have the measure contraction property which reflects the Ricci curvature lower bound in a sense \cite{jui, rif, riz, bar}.
These results lead us to study sub-Riemannian manifolds as an example of the singular Gromov--Hausdorff limit spaces.

In \cite{tas},
the author began to study the topological type of the Gromov--Hausdorff limit space of a sequence of sub-Riemannian manifolds in the following setting.
Let $H_n$ be the $n$-Heisenberg Lie group,
$\h_n$ the associated Lie algebra,
and $\Gamma$ a lattice in $H_n$.
A quatient space $\Gamma\backslash H_n$ is called a \textit{compact Heisenberg manifold}.
Let $\mv$ be a subspace in $\h_n$ and $\langle\cdot,\cdot\rangle$ an inner product on $\mv$.
It induces the left invariant sub-Riemannian structure on $H_n$ (see Example \ref{exlie}).
Since the induced geodesic distance on $H_n$ has the isometric action $\Gamma$ from the left,
thus we obtain a quatient distance on $\Gamma\backslash H_n$ in a canonical way.
We also call such a distance on $\Gamma\backslash H_n$ \textit{left invariant}.

In \cite{tas},
the author studied the non-collapsed limits of compact Heisenberg manifolds with left invariant sub-Riemannian metrics.
Here we say that a sequence is \textit{non-collapsed} if the total measure with respect to the minimal Popp's volume is uniformly bounded below by a positive constant.
The minimal Popp's volume is introduced by the author to compute a family of rank varying sub-Riemannian metrics,
please see Section \ref{secmin}.
Namely we showed that the non-collapsed limit of a sequence of compact Heisenberg manifolds with left invariant sub-Riemannian metrics are again diffeomorphic to a compact Heisenberg manifold of the same dimension.
Since the Heisenberg group can be regarded as a `flat' space in sub-Riemannian geometry,
this result is a sub-Riemannian version of the Mahler's compactness theorem.

In this paper,
we study collapsed Gromov--Hausdorff limits of compact Heisenberg manifolds with left invariant sub-Riemannian metrics.
We say that a sequence of sub-Riemannian manifolds \textit{collapse} if the total measure with respect to the minimal Popp's volume converges to zero.
It complements the above result.

\begin{theorem}[Main result]\label{thmmain}
	Let $\left\{(\Gamma_k\backslash H_n,dist_k)\right\}_{k\in\N}$ be a sequence of compact Heisenberg manifolds endowed with left invariant sub-Riemannian metrics.
Assume that this sequence converges in the Gromov--Hausdorff topology,
the diameter is uniformly bounded above by $D>0$,
and the total measure with respect to the minimal Popp's volume converges to zero.	
Then the limit space is isometric to a flat torus of lower dimension.
\end{theorem}

If the rank of the sub-Riemannian metric on a Heisenberg manifold $\Gamma_k\backslash H_n$ is $2n$,
then the minimal Popp's volume coincides with the Popp's volume (Proposition 5 in \cite{tas}).
This implies the following corollary.

\begin{corollary}
	Let $\left\{(\Gamma_k\backslash H_n,dist_k)\right\}_{k\in\N}$ be a sequence of compact Heisenberg manifolds endowed with left invariant sub-Riemannian metrics of rank $2n$.
Assume that this sequence converges in the Gromov--Hausdorff topology,
the diameter is uniformly bounded above by $D>0$,
and the total measure with respect to the Popp's volume converges to zero.	
Then the limit space is isometric to a flat torus of lower dimension.
\end{corollary}

We explain the idea of the proof of main theorem.
It is well known that a compact Heisenberg manifold has a fiber bundle structure $S^1\to \Gamma\backslash H_n\to \T^{2n}$.
We show that if a sequence collapses,
then the circle fibers also collapse to a point.
Once we show that the fibers collapse,
then the Gromov--Hausdorff limit is isometric to the limit of the base torus with the quotient distances.
It is also known that a Gromov--Hausdorff limit of tori with flat metrics is isometric to a flat torus (Proposition 3.1 in \cite{bet}).
This concludes the theorem.

In the appendix,
we give the systolic inequality on sub-Riemannian compact Heisenberg manifolds with the Popp's volume.
In Riemannian geometry,
the systolic inequality ensures that the existence of the constant $C$ dependent only on the dimension $n$ such that
$$systole\leq C\times measure^{\frac{1}{n}}.$$
It is well known that the topological dimension of a Riemannian manifold is equal to its Hausdorff dimension.

For sub-Riemannian Heisenberg manifolds with the Popp's volume,
We show that the systolic inequality holds with the the Hausdorff dimension exponent (Theorem \ref{thmapp1}).
Moreover,
we give the equality condition in the $3$-dimensional case (Theorem \ref{thmapp2}).

\section*{Acknowledgement}
The author would thank to Prof. Koji Fujiwara for many helpful comments.
The author would thank to Prof. Ryokichi Tanaka for bringing the problem in the appendix to his attention.
This work was supported by JSPS KAKENHI Grant Number JP20J13261.

\section{Preliminaries from sub-Riemannian geometry}\label{sec0}
In this section we prepare notation on sub-Riemannian geometry.
\subsection{Sub-Riemannian structure}

Let $(M,\mathcal{D},g)$ a sub-Riemannian manifold.
We say that a vector field on $M$ is \textit{horizontal} if it is a section of the sub-bundle $\mathcal{D}$.
\begin{example}\label{exlie}
	Let $G$ be a connected Lie group,
	$\g$ the associated Lie algebra,
	$\mv\subset \g$ a subspace and $\langle\cdot,\cdot\rangle$ an inner product on $\mv$.
	For $x\in G$,
	denote by $L_x:G\to G$ the left translation by $x$.
	Define a sub-Riemannian structure on $G$ by
	$$\mathcal{D}_x=(L_x)_{\ast}\mv,~~~~g_x(u,v)=\langle L_{x\ast}^{-1}u,L_{x\ast}^{-1}v\rangle.$$
	Such a sub-Riemannian structure $(\mathcal{D},g)$ is called left invariant.
	We sometimes write left invariant sub-Riemannian structure by $(\mv,\langle\cdot,\cdot\rangle)$.

\end{example}

The associated distance function is given as follows.
We say that an absolutely continuous path $c:[0,1]\to M$ is \textit{admissible} if $\dot{c}(t)\in \mathcal{D}_{c(t)}$ a.e. $t\in [0,1]$.
We define the length of an admissible path by
$$\ell(c)=\int^1_0\sqrt{g(\dot{c}(t),\dot{c}(t))}dt.$$
For $x,y\in M$,
define the distance function by
$$d(x,y)=\inf\left\{\ell(c)~|~c(0)=x,c(1)=y,\text{$c$ is admissible}\right\}.$$
In general not every pair of points in $M$ is joined by an admissible path.
This implies that the value of the function $d$ may be the infinity.
The following \textit{bracket generating} condition ensures that any two points are joined by an admissible path.

\begin{definition}[Bracket generating distribution]\label{dfnbracket}
	For every $i\in\N$,
	let $\mathcal{D}^i$ be the submodule in $Vec(M)$ inductively defined by
	$$\mathcal{D}^1=\mathcal{D},~~~~\mathcal{D}^{i+1}=\mathcal{D}^i+[\mathcal{D},\mathcal{D}^i],$$
	and set $\mathcal{D}_x^i=\left\{X(x)~|~X\in \mathcal{D}^i\right\}$ for $x\in M$.

	We say that a distribution $\mathcal{D}$ is bracket generating if for all $x\in M$ there is $r=r(x)\in\N$ such that $\mathcal{D}_x^r=T_xM$.
\end{definition}

\begin{theorem}[Chow--Rashevskii's theorem, Theorem 3.31 in \cite{agr}]
	Let $(M,\mathcal{D},g)$ be a sub-Riemannian manifold with a bracket generating distribution.
	Then the following two assertions hold.
	\begin{itemize}
		\item[(1)]$(M,d)$ is a metric space,
		\item[(2)]the topology induced by $(M,d)$ is equivalent to the manifold topology.
		\end{itemize}
		In particular,
		$d:M\times M\to \R$ is continuous.
\end{theorem}

Assume that the metric space $(M,d)$ is proper,
that is every closed ball is compact.
With the help of the Ascoli--Alzera theorem,
we can show the existence of a length minimizing path joining any two points (Theorem 3.43 in \cite{agr}).

\subsection{Length minimizing paths}

Let us study the structure of length minimizing paths on a sub-Riemannian manifold $(M,\mathcal{D},g)$.

For simplicity we assume that the dimension of $\mathcal{D}_x$ is equal to $m$ for all $x\in M$ and that we have a family of globally defined $m$ smooth vector fields $\left\{f_1,\dots,f_m\right\}$ such that $\left\{f_1(x),\dots,f_m(x)\right\}$ is an orthonormal basis of $(\mathcal{D}_x,g_x)$.
We call such a family a \textit{generating family}.

Let $\sigma$ be the canonical symplectic form on $T^{\ast}M$.
For a given function $h:T^{\ast}M\to \R$,
there is a unique vector field $\vec{h}$ on $T^{\ast}M$ defined by
$$\sigma(\cdot,\vec{h}(\lambda))=dh|_{\lambda}~~~~(\lambda\in T^{\ast}M).$$
We call the vector $\vec{h}$ the \textit{Hamiltonian vector field of $h$}.

Let $h_i:T^{\ast}M\to \R$,
$i=1,\dots,m$ be the function defined by
$$h_i(\lambda)=\langle\lambda|f_i(x)\rangle~~~~(\lambda\in T^{\ast}_xM),$$
where $\langle\cdot|\cdot\rangle$ is the canonical pairing of covectors and vectors.
Then length minimizing paths on a sub-Riemannian manifold are explained with the Hamiltonian vector fields of $h_i$'s as follows.

\begin{theorem}[The Pontryagin maximal principle, Theorem 4.20 in \cite{agr}]\label{thmextremal}
	Let $\gamma:[0,T]\to M$ be a length minimizing path parametrized by constant speed,
	and write its differential by
	$$\dot{\gamma}(t)=\sum_{i=1}^m u_i(t)f_i(\gamma(t)).$$
	Then there is a Lipschitz curve $\lambda:[0,T]\to T^{\ast}M$ such that
	\begin{equation}\label{eqlambda}
		\begin{cases}
		\lambda(t)\in T^{\ast}_{\gamma(t)}M,\\
		\dot{\lambda}(t)=\sum_{i=1}^m u_i\vec{h}_i(\lambda(t))~~~~a.e.~t\in[0,T],
	\end{cases}\end{equation}
	and one of the following conditions satisfied:
	\begin{itemize}
		\item[(N)]$h_i(\lambda(t))=u_i(t),~~~~i=1,\dots,m,~~t\in[0,1]$,
		\item[(A)]$h_i(\lambda(t))=0,~~~~i=1,\dots,m$.
		\end{itemize}
\end{theorem}

\begin{definition}[Normal extremal]
	The Lipschitz curve $\lambda$ with the condition (N) is called a normal extremal,
and its projection $\gamma$ is called a normal trajectory.
\end{definition}
\begin{definition}[Abnormal extremal]
	The Lipschitz curve $\lambda$ with the condition (N) is called an abnormal extremal,
	and its projection $\gamma$ is called an abnormal trajectory.
\end{definition}

Let $\mathcal{D}^{\bot}_x\subset T_x^{\ast}M$ be the subspace defined by
$$\mathcal{D}^{\bot}_x=\left\{\lambda\in T_x^{\ast}M\mid \langle\lambda|u\rangle=0~~\text{for all}~u\in \mathcal{D}_x\right\}.$$
With this notation,
we can say that an extremal $\lambda(t)=(x(t),p(t))$ is abnormal if $\lambda(t)\in \mathcal{D}^{\bot}_{x(t)}$.

\begin{example}
	Suppose that $(M,\mathcal{D},g)$ is a Riemannian manifold.
	Then the condition (A) implies that $\lambda(t)\in \mathcal{D}^{\bot}_{x(t)}=\left\{0\right\}$.
Combined with the second equality in (\ref{eqlambda}),
such a Lipschitz curve $\lambda$ is a constant curve.
This argument shows that every abnormal trajectories on a Riemannian manifold is a constant curve.
\end{example}

\begin{remark}\label{rmkabnormalnormal}
	\begin{itemize}
		\item An abnormal extremal and a normal extremal may project to the same trajectory.
	Hence a trajectory may be normal and abnormal simultaneously.

	\item On the Heisenberg group and its quotient spaces,
	there are no non-trivial abnormal extremals.
	Thus we compute only normal extremals in this paper.
\end{itemize}
\end{remark}
	A normal extremal is the solution to the differential equation,
called the Hamiltonian system.
Let $H:T^{\ast}M\to \R$ be the function defined by 
$$H(\lambda)=\frac{1}{2}\sum_{i=1}^mh_i(\lambda)^2.$$
This function is called the \textit{sub-Riemannian Hamiltonian}.

\begin{theorem}[Theorem 4.25 in \cite{agr}]\label{thmnormalgeodesic}
	A Lipschitz curve $\lambda:[0,T]\to T^{\ast}M$ is a normal extremal if and only if it is a solution to the Hamiltonian system
	$$\dot{\lambda}(t)=\vec{H}(\lambda(t)),~~~~t\in[0,T],$$
	where $\vec{H}$ is the Hamiltonian vector field of $H$.

	Moreover,
the corresponding normal trajectory $\gamma$ is smooth and has a constant speed satisfying
$$\frac{1}{2}\|\dot{\gamma}(t)\|^2=H(\lambda(t)).$$
\end{theorem}

A straightforward computation shows the local expression for $\lambda(t)=(x(t),p(t))$ written by
\begin{equation}\label{hamiltonianform}
	\begin{cases}
	\dot{x}(t)=\frac{\partial H}{\partial p},\\
	\dot{p}(t)=-\frac{\partial H}{\partial x}.
\end{cases}\end{equation}

\subsection{The Popp's volume}\label{secpopps}

On a Riemannian manifold,
one has a canonical Riemannian volume form defined by
$$v=\nu_1\wedge\cdots\wedge\nu_n,$$
where $\{\nu_1,\cdots\nu_n\}$ is a dual coframe of an orthonormal basis.
In sub-Riemannian geometry,
we also have a canonical volume form,
called \textit{Popp's volume} introduced in \cite{mon}.
The Popp's volume is defined under the following \textit{equiregular} assumption.
\begin{definition}[Equiregular distribution]
	A sub-Riemannian manifold $(M,\mathcal{D},g)$ is equiregular if for any $i\in N$ the dimension of the subspaces $\mathcal{D}_x^i$ is independent of the choice of $x\in M$.
\end{definition}
	If $\mathcal{D}_x^r=T_xM$,
	we say that a sub-Riemannian manifold is $r$-step.
For simplicity,
we recall the definition of the Popp's volume in the $2$-step case.

\begin{definition}[Nilpotentization]
	The nilpotentization of $\mathcal{D}$ at the point $x\in M$ is the graded vector space
	$$gr_x(\mathcal{D})=\mathcal{D}_x\oplus \mathcal{D}_x^2/\mathcal{D}_x.$$
\end{definition}

On the vector space $gr_x(\mathcal{D})$ we can define a new Lie bracket $[\cdot,\cdot]'$ by
$$[X~mod~\mathcal{D},Y~mod~\mathcal{D}]_x^{\prime}=[X,Y]_x~\mod~\mathcal{D}_x.$$
The new Lie bracket rule induces a different Lie algebra structure from the original one.

From the inner product on $\mathcal{D}_x$,
we obtain the inner product on the nilpotentization $gr_x(\mathcal{D})$ of $\mathcal{D}$.
Let $\pi:\mathcal{D}_x\otimes \mathcal{D}_x\to \mathcal{D}_x^2/\mathcal{D}_x$ be the linear map given by
$$\pi(u\otimes v)=[U,V]_x~~mod ~~\mathcal{D}_x,$$
where $U,V$ are horizontal extensions of $u,v$.
Define the norm $\|\cdot\|_2$ on $\mathcal{D}^2_x/\mathcal{D}_x$ by
$$\|z\|_2=\min\left\{\|U(x)\|\|V(x)\|~|~[U,V]_x=z ~mod~\mathcal{D},~U,V:~\text{horizontal vector fields} \right\}.$$
This norm satisfies the parallelogram law,
thus we obtain the inner product $\langle\cdot,\cdot\rangle_2$ on $\mathcal{D}_x^2/\mathcal{D}_x$.
The direct sum of two inner product spaces $(\mathcal{D}_x,g_x)$ and $(\mathcal{D}^2_x/\mathcal{D}_x,\langle\cdot,\cdot\rangle_2)$ gives the new inner product $\langle\cdot,\cdot\rangle_x^{\prime}$ on the nilpotentization $gr_x(\mathcal{D})$.

Let $\omega_x\in \wedge^ngr_x(\mathcal{D})^{\ast}$ be the volume form obtained by wedging the elements of orthonormal dual basis in $(gr_x(\mathcal{D}),\langle\cdot,\cdot\rangle_x^{\prime})$.
It is defined up to sign.
By the following lemma,
the volume $\omega_x\in \wedge^ngr_x(\mathcal{D})^{\ast}$ is transported to the volume on $\wedge^nT_x^{\ast}M$.

\begin{lemma}[Lemma 10.4 in \cite{mon}]\label{lemfilt}
	Let $E$ be a vector space of dimension $n$ with a filtration by linear subspaces $F_1\subset F_2\subset\cdots F_l=E$.
	Let $Gr(F)=F_1\oplus F_2/F_1\oplus\cdots F_l/F_{l-1}$ be the associated graded vector space.
	Then there is a canonical isomorphism $\theta:\wedge^n E^{\ast}\simeq \wedge^n gr(F)^{\ast}$.
\end{lemma}

Let $\theta:\wedge^n T_x^{\ast}M\to \wedge^ngr_x(\mathcal{D})^{\ast}$ be the isomorphism obtained by Lemma \ref{lemfilt}.

\begin{definition}[Popp's volume]\label{dfnpopp}
	The Popp's volume form $vol(\mathcal{D},g)$ is defined by
	$$vol(\mathcal{D},g)_x=\theta^{\ast}\omega_x,~~~~x\in M.$$
\end{definition}

\noindent Trivially the Popp's volume of a Riemannian manifold is the canonical Riemannian volume form.

The Popp's volume has a useful expression by using the structure constant.
We say that a local frame $X_1,\dots,X_n$ is adapted if $X_1,\dots,X_m$ are orthonormal.
Define the smooth functions $c_{ij}^l$ on $M$ by
$$[X_i,X_j]=\sum_{l=1}^nc_{ij}^lX_l.$$
We call them the \textit{structure constants}.
We define the $n-m$ dimensional square matrix $B$ by
$$B_{hl}=\sum_{i,j=1}^m c_{ij}^h c_{ij}^l.$$

\begin{theorem}[Theorem 1 in \cite{barriz}]\label{thmpopp}
	Let $X_1,\dots,X_n$ be a local adapted frame,
	and $\nu^1,\dots,\nu^n$ the dual coframe.
	Then the Popp's volume $vol(\mathcal{D},g)$ is written by
	$$vol(\mathcal{D},g)=\left(\det B\right)^{-\frac{1}{2}}\nu^1\wedge\cdots\wedge \nu^n.$$
\end{theorem}

\subsection{The minimal Popp's volume form on Lie groups}\label{secmin}
Let $G$ be a connected Lie group,
$\g$ its Lie algebra,
and  $(\mv,\langle\dot,\cdot\rangle)$ a left invariant sub-Riemannian structure on $G$.
Moreover let $\mathcal{F}(\mv)$ be the set of bracket generating subspaces in $\mv$.
We define the minimal Popp's volume on $(G,\mv,\langle\cdot,\cdot\rangle)$ as follows.

\begin{definition}\label{dfnminvol}
	The minimal Popp's volume is defined by
	$$mvol(\mv,\langle\cdot,\cdot\rangle)=\min\{vol(\mw,\langle\cdot,\cdot\rangle|_{\mw\otimes\mw})\mid \mw\subset\mathcal{F}(\mv)\}.$$
\end{definition}
Here the order of volume forms is obtained from the coefficients of a fixed Haar volume $vol_0$.
Hence we can take the infimum up to sign.
The existence of the minimum is shown as follows.
For a positive constant $C>0$,
define a closed subset $\mathcal{F}(\mv,C)\subset\mathcal{F}(\mv)$ by
$$\mathcal{F}(\mv,C)=\{\mw\subset\mathcal{F}(\mv)\mid |vol(\mw,\langle\cdot,\cdot\rangle|_{\mw\otimes\mw})|\leq |Cvol_0|\}.$$
From its definition,
$\mathcal{F}(\mv,C)$ is a closed subset of a finite union of the Grassmannians,
thus the minimum exists from its compactness.

\section{Compact Heisenberg manifolds}

In this section,
we recall fundamental properties on compact Heisenberg manifolds.

\subsection{The moduli space of compact Heisenberg manifolds}

In this section,
we recall a useful parametrization of the isometry classes of left invariant sub-Riemannian metrics on a compacgt Heisenberg manifold $\Gamma\backslash H_n$.

\begin{definition}
	For a lattice $\Gamma<H_n$ and $k\in\N$,
	we denote by $\mathcal{M}_k(\Gamma\backslash H_n)$ the set of isometry classes of left invariant sub-Riemannian metrics of corank $k$ on $\Gamma\backslash H_n$.

	We call the union $\mathcal{M}(\Gamma\backslash H_n):=\bigcup_{k\in\N}\mathcal{M}_k(\Gamma\backslash H_n)$ the moduli space of left invariant sub-Riemannian metrics on $\Gamma\backslash H_n$.
\end{definition}

To obtain an explicit form of the moduli space,
we recall the classification of lattices in $H_n$.
Fix a basis $\{X_1,\dots,X_{2n},Z\}$ of $\h_n$ such that $[X_i,X_{i+n}]=Z$ and the other brackets are zero.
Let $D_n$ be the set of $n$-tuples of integers $\bm{r}=(r_1,\dots,r_n)$ such that $r_i$ divides $r_{i+1}$ for all $i=1,\dots,n$.
For $\bm{r}\in D_n$,
let $\Gamma_{\bm{r}}< H_n$ be the subgroup defined by
$$\Gamma_{\bm{r}}=\langle\exp(r_1X_1),\dots,\exp(r_nX_n),\exp(X_{n+1}),\dots,\exp(X_{2n}),\exp(Z)\rangle.$$
This gives a characterization of lattices in the Heisenberg Lie group.

\begin{theorem}[Theorem 2.4 in \cite{gor2}]\label{thm2-2.1}
	Any uniform lattice $\Gamma<H_n$ is isomorphic to $\Gamma_{\bm{r}}$ for some $\bm{r}\in D_n$.

Moreover,
$\Gamma_{\bm{r}}$ is isomorphic to $\Gamma_{\bm{s}}$ if and only if $\bm{r}=\bm{s}$.
\end{theorem}

Fix a lattice $\Gamma_{\bm{r}}$,
and consider the moduli space $\mathcal{M}(\Gamma_{|\bm{r}}\backslash H_n)$.
Let $G_{\bm{r}}$ be the set of matrices in $GL_{2n}(\R)$ given by
$$G_{\bm{r}}=diag(\bm{r})GL_{2n}(\Z)diag(\bm{r})^{-1},$$
where $diag(\bm{r})$ is a diagonal matrix with its diagonal entries $r_1,\dots,r_n,1,\dots,1$.
Let $J_n$ be the skew-symmetric matrix given by
$$J_n=\begin{pmatrix}
	O & I_n\\
	-I_n & O
\end{pmatrix},$$
where $I_n$ is the identity matrix.
We embed the group
$$\widetilde{Sp}(2n,\R)=\left\{\beta\in GL_{2n}(\R)\mid \beta J_n
	\beta=\epsilon(\beta)J_n
,~\epsilon(\beta)=\pm 1 \right\}$$
into $GL_{2n+1}(\R)$ via the mapping $\iota:\beta\mapsto \begin{pmatrix}
	\beta & 0\\
	0 & \epsilon(\beta)
\end{pmatrix}$.

Define the set of matrices $\Pi_{\bm{r}}\subset GL_{2n+1}(\R)$ by
$$\Pi_{\bm{r}}=\iota(G_{\bm{r}}\cap \widetilde{Sp}(2n,\R)).$$
Moreover let
$$\mathcal{A}=\left\{A=\begin{pmatrix}
	\tilde{A} & 0\\
	0 & \rho_A
\end{pmatrix}~\Bigg{|}~\tilde{A}\in GL_{2n}(\R),\rho_A\in\R\right\},$$
and
$$\mathcal{R}=\left\{R=\begin{pmatrix}
	\tilde{R} & 0\\
	0 & \pm 1
\end{pmatrix}~\Bigg{|}~\tilde{R}\in O(2n)\right\}.$$
With these notations,
we can parametrize $\mathcal{M}(\Gamma_{\bm{r}}\backslash H_n)$ as follows.

\begin{definition}[Theorem 3.6 in \cite{tas}]\label{dfn2-3}
	Isometry classes of a compact Heisenberg manifold $\Gamma_{\bm{r}}\backslash H_n$ with left invariant sub-Riemannian metrics of various corank are parametrized by
	$$\Pi_{\bm{r}}\backslash \mathcal{A}/\mathcal{R}.$$
\end{definition}

This identification is given as follows.
Let $\mv_0=Span\{X_1,\dots,X_{2n}\}\subset\h_n$.
Identify a matrix $A=\begin{pmatrix}
	\tilde{A} & 0\\
	0 & \rho_A\end{pmatrix}\in \mathcal{A}$ to a linear endomorphism on $\h_n$ in the basis $\left\{X_1,\dots,X_{2n},Z\right\}$.
Then the image of the matrix $A$ is $\h_n$ (resp. $\mv_0$) if $\rho_A\neq 0$ (resp. $\rho_A=0$).
When $\rho_A\neq 0$ (resp. $\rho_A=0$),
define the inner product $\langle\cdot,\cdot\rangle_A$ on $\h_n$ (resp. $\mv_0$) such that its orthonormal basis is $\{AX_1,\dots,AZ\}$ (resp. $\{AX_1,\dots,AX_{2n}\}$).
In both cases,
a matrix $A\in\mathcal{A}$ defines a left invariant (sub-)Riemannian metric on a compact Heisenberg manifold $\Gamma_{\bm{r}}\backslash H_n$.
On the other hand,
any left invariant sub-Riemannian metric is isometric to such a metric.
The induced metric spaces are isometric if and only if the matrices have the same representatives in $\Pi_{\bm{r}}\backslash \mathcal{A}\mathcal{R}$.

\subsection{The operator $j$}
We recall the operator $j:\mathcal{A}\to End(\mv_0)$,
which has an important role in the study of nilpotent Lie groups.

Let $Z^{\ast}\in\h_n^{\ast}$ be the dual covector of the vector $Z\in[\h_n,\h_n]\subset\h_n$.
For a matrix $A\in\mathcal{A}$,
define a skew symmetric operator $j(A):\mv_0\to \mv_0$ by
$$\langle j(A)(X),Y\rangle_A=Z^{\ast}([X,Y]).$$
The operator $j(A)$ has the following matrix representation.

\begin{lemma}[Lemma 4 in \cite{tas}]
	The operator $j(A)$ has a matrix representation
	$$j(A)=\hspace{1pt}^{t}\hspace{-2pt}\tilde{A}J_n\tilde{A}$$
	in the basis $\left\{AX_1,\dots,AX_{2n}\right\}$.
\end{lemma}

	It is a fundamental fact on skew-symmetric matrices that by mutiplicating an appropriate orthogonal matrix from the right of $\tilde{A}$,
	we can let the matrix representation
	\begin{equation}\label{eq3-0}
		\hspace{1pt}^{t}\hspace{-2pt}\tilde{A}J_n\tilde{A}=\begin{pmatrix}O & diag(d_1(A),\dots,d_n(A))\\
			-diag(d_1(A),\dots,d_n(A)) & O
		\end{pmatrix}.
	\end{equation}
	Here $d_i(A)$'s are positive numbers such that $\pm\sqrt{-1}d_1,\dots,\pm\sqrt{-1}d_n$ are the eigenvalues of $\hspace{1pt}^{t}\hspace{-2pt}\tilde{A}J_n\tilde{A}$.

	\begin{definition}\label{dfn3-1}
		We call a matrix $A\in\mathcal{A}$ canonical form if it satisfies the equation (\ref{eq3-0}).
	\end{definition}

We will assume $d_1\leq\cdots\leq d_n$.
It is easy to see that the functions $d_i$'s are well posed on the moduli space $\mathcal{M}(\Gamma_{\bm{r}}\backslash H_n)$,
please see Lemma 5 in \cite{tas}.

The positive number $d_n$ can be thought as the $\ell^{\infty}$-norm of the matrix $\hspace{1pt}^{t}\hspace{-2pt}\tilde{A}J_n\tilde{A}$ as an element in the Euclidean space $\R^{4n^2}$.
We also mension its $\ell^2$-norm,
the Hilbert--Schmidt norm of matrices.

\begin{definition}\label{dfn3-2}
	For a matrix $A\in\mathcal{A}$,
	we define $\delta(A)=\|\hspace{1pt}^{t}\hspace{-2pt}\tilde{A}J_n\tilde{A}\|_{HS}$.
\end{definition}

The function $\delta$ is also well posed on the moduli space $\mathcal{M}(\Gamma_{\bm{r}}\backslash H_n)$,
please see Lemma 6 in \cite{tas}.
The following lemma is useful for later calculations.

\begin{lemma}[Lemma 11 in \cite{tas}]\label{lemfundamental}
	For a matrix $A\in\mathcal{A}$,
	we have
	\begin{itemize}
		\item[(1)]$\delta(A)=\sqrt{2\sum_{i=1}^nd_i(A)^2}$,
		\item[(2)]$|\det(\tilde{A})|=\prod_{i=1}^nd_i(A)$.
		\end{itemize} 
\end{lemma}

\subsection{Geodesics}
Let $A\in\mathcal{A}$ be a matrix of canonical form.
For $i=1,\dots,n$,
define the functions $h_{x_i}$ (resp. $h_{y_i}$ and $h_z$)$:T^{\ast}H_n\to\R$ by $h_{x_i}(p)=p(AX_i(x))$ (resp. $h_{y_i}(p)=p(AY_i(x))$ and $h_z(p)=p(Z(x))$) for $p\in T^{\ast}_x H_n$.
Suppose that an admissible path $\gamma:[0,T]\to H_n$ parametrized by $\gamma(t)=\exp(\sum_{i=1}^{n}x_i(t)AX_i+y_i(t)AY_i+z(t)Z)$ is length minimizing.
By Theorem \ref{thmnormalgeodesic} and Remark \ref{rmkabnormalnormal},
there is a lift $\ell:[0,T]\to T^{\ast}H_n$ of $\gamma$ such that the following Hamiltonian equation holds.
\begin{align*}
	\begin{cases}
		\dot{h}_{x_i}=\lambda_i h_zh_{y_i} & (i=1,\dots,n),\\
		\dot{h}_{y_i}=-\lambda_i h_zh_{x_i} & (i=1,\dots,n),\\
		\dot{h}_z=0,\\
		\dot{x}_i=h_{x_i}, & (i=1,\dots,n)\\
		\dot{y}_i=h_{y_i}, & (i=1,\dots,n)\\
		\dot{z}=\frac{1}{2}\sum_{i=1}^nd_i(A)\left(x_ih_{y_i}-y_{i}h_{x_i}\right)+\rho_A^2p_z,
	\end{cases}
\end{align*}
\noindent where we write $h_{x_i}(t)=h_{x_i}\circ\ell(t)$,
$h_{y_i}(t)=h_{y_i}\circ\ell(t)$ and $h_z(t)=h_z\circ\ell(t)$.

\begin{lemma}[Proposition 3.5 in \cite{ebe} for Riemannian case and Lemma 14 in \cite{riz} for sub-Riemannian case]\label{lemgeod}

	Let $\gamma:[0,T]\to H_n$ be the geodesic issuing from the identity with the initial data of the extremal
	$$\left(h_{x_1}(0),\dots,h_{y_1}(0),\dots,h_z(0)\right)=(p_{x_1},\dots,p_{y_1},\dots,p_z).$$
	Set $\xi_i=p_zd_i(A)$.
	Then $\gamma$ is parametrized as follows.
\begin{itemize}
	\item If $p_z\neq 0$,
		then
$$
	\begin{pmatrix}
		x_{i}(t)\\
	  y_i(t)\end{pmatrix}
	=\frac{1}{\xi_i}\begin{pmatrix}\sin(\xi_it) & \cos(\xi_it)-1\\
		-\cos(\xi_it)+1 & \sin(\xi_i t)
		\end{pmatrix}\begin{pmatrix}
		p_{x_i}\\
		p_{y_i}
	\end{pmatrix}
$$
for each $i=1,\dots,n$,
and
$$
	z(t)=\rho_A^2p_zt+\frac{1}{2}\sum_{i=1}^{n}\left(\frac{\lambda_it}{\xi_i}-\frac{\lambda_i}{\xi_i^2}\sin(\xi_it)\right)\left(p_{x_i}^2+p_{y_i}^2\right).
$$
\item If $p_z=0$,
	then $x_i(t)=p_{x_i}t$,
	$y_i(t)=p_{y_i}t$ and $z(t)\equiv 0$.
\end{itemize}

\end{lemma}

For later arguments,
we give an explicit distance from the identity to points in the horizontal direction and the vertical direction.
Denote by $\widetilde{dist}_A$ the distance function on $H_n$ associated to $A\in\mathcal{A}$.
\begin{lemma}[Proposition 3.11 in \cite{ebe} for Riemannian case and Lemma 9 in \cite{tas} for general cases]\label{lemhorizontal}
	For $U\in\mv_0$ and $V\in[\h_n,\h_n]$,
we have
$$\widetilde{dist}_A(e,\exp(U+V))\geq\|U\|_A.$$
Moreover,
the equality holds if and only if $V=0$.
\end{lemma}

\begin{lemma}[Lemma 10 in \cite{tas}]\label{lemvertical}
	For $p\in\R$,
	the distance from $e$ to $\exp(pZ)$ is given by
	$$\widetilde{dist}_A(e,\exp(pZ))=\min\left\{\left|\frac{p}{\rho_A}\right|,\frac{2}{d_n(A)}\sqrt{|p|\pi d_n(A)-\pi^2\rho_A^2}\right\},$$
where we set $\left|\frac{p}{\rho_A}\right|=+\infty$ if $\rho_A=0$ and $\sqrt{|p|\pi d_n(A)-\pi^2\rho_A^2}=\infty$ if $|p|\pi d_n(A)\leq \pi^2\rho_A^2$.
	\end{lemma}

\subsection{Volume forms}
In this section,
we recall an explicit formula of the Riemannian volume form,
the Popp's volume form and the minimal Popp's volume on the Heisenberg Lie group.

For a matrix $A\in\mathcal{A}$ with $\rho_A\neq 0$,
denote by $v(\h_n,A)$ the Riemannian volume form.
Since it is the wedge of the dual coframe of an orthonormal frame,
we have
\begin{equation}\label{eqriemvol}
	v(\h_n,A)=\rho_A^{-1}(\det \tilde{A})^{-1}X_1^{\ast}\wedge\cdots\wedge X_{2n}^{\ast}\wedge Z^{\ast}.
\end{equation}

Next let $A\in\mathcal{A}$ be a matrix with $\rho_A=0$.
Denote by $v(\mv_0,A)$ the Popp's volume associated to the sub-Riemannian structure $(\mv_0,\langle\cdot,\cdot\rangle_A)$.

Since $\hspace{1pt}^{t}\hspace{-2pt}\tilde{A}J_n\tilde{A}$ is the matrix representation of $j_A$ in the basis $\left\{AX_1,\dots,AX_{2n}\right\}$,
its $(i,j)$-th entry coincides with the structure constant $c_{ij}=Z^{\ast}([AX_i,AX_j])$.
By Theorem \ref{thmpopp},
the Popp's volume $v(\mv_0,A)$ is written by
\begin{align}\label{eqpoppvol}
	v(\mv_0,A)=\delta(A)^{-1}(\det\tilde{A})^{-1}X_1^{\ast}\wedge\cdots\wedge X_{2n}^{\ast}\wedge Z^{\ast}.
\end{align}

The explicit formula of the minimal Popp's volume form is given in \cite{tas}.
After a straight forward calculation,
one finds that it is the minimum of the above two forms (\ref{eqriemvol}) and (\ref{eqpoppvol}).
	\begin{proposition}[Proposition 5 in \cite{tas}]\label{propminvol}
		For a matrix $A\in\mathcal{A}_0\cup\mathcal{A}_1$ of canonical form,
		$$v(A)=\min\{|\rho_A|^{-1},\delta(A)^{-1}\}|\det(\tilde{A})|^{-1}X_1^\ast\wedge\cdots\wedge Z^{\ast},$$
		where we write $|\rho_A|^{-1}=\infty$ if $\rho_A=0$.
	\end{proposition}

	\subsection{The circle bundle structure}\label{sec9}

	Fix a $n$-tuple of numbers $\bm{r}\in D_n$.
We recall a circle bundle structure of a compact Heisenberg manifold $\Gamma_{\bm{r}}\backslash H_n$.
Let $P_0:H_n\to\h_n\to\mv_0$ be the composition of the logarithem map and the projection.	
Then one obtains a surjective map $\overline{P}_0:\Gamma_{\bm{r}}\backslash H_n\to P_0(\Gamma_{\bm{r}})\backslash \mv_0$ such that the following diagram is commutative.

	\begin{equation}\label{diag1}
		\begin{diagram}
		\node{H_n}
		\arrow{e,t,2}{P_0}
		\arrow{s,l,2}{P_{\Gamma_{\bm{r}}}}
		\node{\mv_0}
		\arrow{s,l,2}{P_{\Z_{\bm{r}}}}\\
		\node{\Gamma_{\bm{r}}\backslash H_n}
		\arrow{e,t,2}{\overline{P}_0}
		\node{P_0(\Gamma_{\bm{r}})\backslash \mv_0}.
	\end{diagram}
\end{equation}
Here the vertical arrows are the quotient map.
The compact Heisenberg manifold $\Gamma_{\bm{r}}\backslash H_n$ has a circle bundle structure by this map $\overline{P}_0$.
For each $b\in P_0(\Gamma_{\bm{r}})\backslash \mv_0$,
	we denote by $F_b$ the fiber over $b$.

	Denote by $\widetilde{dist}_A$ the left invariant sub-Riemannian distance on $H_n$ associated to a matrix $A\in\mathcal{A}$,
	$dist_A$ the quotient distance on $\Gamma_{\bm{r}}\backslash H_n$,
	and $\widetilde{dist}_{\tilde{A}}$ the quotient distance on $\mv_0$.
	Notice that the quotient distance $\widetilde{dist}_{\tilde{A}}$ is induced from the inner product $\langle\cdot,\cdot\rangle_{\tilde{A}}$ on $\mv_0$ whose orthonormal basis is $\{\tilde{A}X_1,\dots,\tilde{A}X_{2n}\}$.

	On the base space $P_0(\Gamma_{\bm{r}})\backslash \mv_0$,
	we can define the quotient distance $dist_{\tilde{A}}$ by the quotient maps $\overline{P}_0$ or $P_{\Z_{\bm{r}}}$,
	since the diagram (\ref{diag1}) is commutative.
	This quotient distance is induced from a flat Riemannian metric,
	denoted by $g_{\tilde{A}}$.

From these arguments,
we can say that the limit of a sequence of compact Heisenberg manifolds is isometric to that of base tori with the quotient flat metrics if their circle fibers converge to a point.
As we see in the following proposition,
we know that every Gromov--Hausdorff limit of a sequence of flat tori is again a torus.

\begin{proposition}[Proposition 3.1 in \cite{bet}]\label{prop9-1}
	Let $\left\{(\T^n,g_k)\right\}_{k\in\N}$ be a sequence of flat tori which converges to $(X,d_X)$ in the Gromov--Hausdorff topology.
	Then $(X,d_X)$ is isometric to a flat torus $(\T^m,g_{\infty})$ for some $m\leq n$.
\end{proposition}

\section{The conditions for fibers collapsing to a point}\label{sec10}
Let $\left\{(\Gamma_{\bm{r}(k)}\backslash H_n,dist_k)\right\}$ be a sequence of compact sub-Riemannian Heisenberg manifolds with the diameter upper bound by $D>0$.
	In this section,
	we show that if the sequence collapses,
	then the diameter of the circle fibers converge to zero.

	The fiber over $b\in \Gamma_{\bm{r}}\backslash H_n $ is written by $F_b=\left\{\Gamma_{\bm{r}(k)}\exp(\R Z)h_b\right\}$,
	where we fix $h_b\in P_{\Gamma_{\bm{r}(k)}}^{-1}\left(\overline{P}_0^{-1}\left(b\right)\right)\subset H_n$.
		In particular,
		the subset $\left\{\exp(tZ)h_b\mid t\in[0,1)\right\}\subset H_n$ is a representative of $F_b$.
		By the homogeneity of the restricted distance on $F_b$,
		its diameter is the half of the distance from $h_b$ to $\exp(Z)h_b$ and is independent of the choice of $h_b$.
Moreover,
again by the homogeneity of the distance on $H_n$,
			the diameter of the fiber $F_b$ is independent of the choice of a point $b$ in the base $P_0(\Gamma_{\bm{r}})\backslash H_n$.

			The above argument shows the following lemma.
			\begin{lemma}\label{lemfiberdiam}
				The diameter of the fibers $F_b$ is given by
				$$diam(F_b)=\widetilde{dist}\left(e,\exp\left(\frac{1}{2}Z\right)\right).$$
			\end{lemma}

		Let us pass to the estimate of the diameter.
		First we consider a sequence $\left\{\left(\Gamma_{\bm{r}(k)},dist_{A_k}\right)\right\}_{k\in\N}$ such that $\bm{r}(k_1)\neq \bm{r}(k_2)$ for any $k_1\neq k_2$.
		This implies that the sequence of numbers $\{r_n(k)\}$ diverges.

	\begin{lemma}\label{lem9-2}
		Assume that $r_n(k)$ diverge to infinity.
		Then the diameters of the fibers $F_b$ converge to zero.
	\end{lemma}

	\begin{proof}
		Put $\gamma_{n,k}=\exp(\frac{r_n(k)}{2}X_n)$.
		Since $\gamma_{n,k}$ is on the plane $\exp(\mv_0)$,
		a length minimizing path from $e$ to $\gamma_{n,k}$ in $H_n$ is the straight segment $\exp(sX_n)$,
		$s\in [0,\frac{r_n(k)}{2}]$.
		Moreover its projection by $P_{\Gamma_{\bm{r}(k)}}$ is a length minimizing path from $\Gamma_{\bm{r}(k)} e$ to $\Gamma_{\bm{r}(k)}\gamma_{n,k}$.
Indeed,
any element in $\Gamma_{\bm{r}(k)}\gamma_{n,k}$ is written by
$$\exp\left(r_n(k)\left(z+\frac{1}{2}\right)X_n+P\right),$$
where $z\in\Z$ and $P$ is a element in $\h_n$ transverse to $X_n$.
Clearly a length minimizing path from $\Gamma_{\bm{r}(k)} e$ to $\Gamma_{\bm{r}(k)} \gamma_{n,k}$ is realized when
$$z=0,-1 ~~~\text{and}~~~P=0.$$
This shows that the projection of the straight segment $\exp(sX_n)$ is length minimizing.

Since the length of the straight segment $\exp(sX_n)$ is $\left\|\frac{r_n(k)}{2}X_n\right\|_{A_k}$,
we obtain
\begin{align}\label{eq5-11}
	\left\|\frac{r_n(k)}{2}X_n\right\|_{A_k}=\widetilde{dist}_{A_k}\left(e,\gamma_{n,k}\right)
	&=dist_{A_k}\left(\Gamma_{\bm{r}(k)} e,\Gamma_{\bm{r}(k)} \gamma_{n,k}\right)\\
	&\leq diam\left(\Gamma_{\bm{r}(k)}\backslash H_n,dist_{A_k}\right)\\
	&\leq D.
		\end{align}

		By the same argument we also show that
		\begin{equation}\label{eqx2n}
			\left\|\frac{X_{2n}}{2}\right\|_{A_k}\leq D.
		\end{equation}
	
		On the other hand,
		let $c:[0,4]\to H_n$ be a path inductively defined by
		$$c(t)=\begin{cases}
			\exp(-t\sqrt{\frac{r_n(k)}{2}}X_n) & t\in[0,1],\\
			c(1)\exp(-(t-1)\frac{1}{\sqrt{2r_n(k)}}X_{2n}) & t\in[1,2],\\
			c(2)\exp((t-2)\sqrt{\frac{r_n(k)}{2}}X_n) & t\in[2,3],\\
			c(3)\exp((t-3)\frac{1}{\sqrt{2r_n(k)}}X_{2n}) & t\in[3,4].
		\end{cases}
	$$
	The endpoint of $c$ is $c(4)=\exp(\frac{1}{2}Z)$,
	and the length is computed as
	\begin{align*}
		length(c)&=\|\sqrt{2r_n(k)}X_n\|_{A_k}+\left\|\sqrt{\frac{2}{r_n(k)}}X_{2n}\right\|_{A_k}\\
			 &=\sqrt{2r_n(k)}\|X_n\|_{A_k}+\sqrt{\frac{2}{r_n(k)}}\|X_{2n}\|_{A_k}\\
			 &\leq \sqrt{2r_n(k)}\frac{2D}{r_n(k)}+\sqrt{\frac{2}{r_n(k)}}2D\\
			 &=\frac{4\sqrt{2}D}{\sqrt{r_n(k)}}.
	\end{align*}
		Here the third inequality follows from (\ref{eq5-11}) and (\ref{eqx2n}).
		Hence we obtain
		$$\tilde{d}_{A_k}\left(e,\exp\left(\frac{1}{2}Z\right)\right)\leq length (c)\leq \frac{4\sqrt{2}D}{\sqrt{r_n(k)}}.$$
		Since $r_n(k)$ diverges to the infinity,
		the diameters of the fibers converge to zero.
	\end{proof}

	By Lemma \ref{lem9-2},
	we only need to consider a sequence consisting of a fixed diffeomorphism type $\Gamma_{\bm{r}}\backslash H_n$.
In the following proposition we fix a diffeomorphism type.

	\begin{proposition}\label{prop9-2}
	Let $\{\Gamma_{\bm{r}}\backslash H_n,dist_{A_k}\}$ be a sequence of compact Heisenberg manifolds with left invariant sub-Riemannian metrics.
	If the total measure in the minimal Popp's volume converges to zero,
	then the diameter of the fibers converges to zero.
	\end{proposition}

	\begin{proof}
		By Proposition \ref{propminvol},
		the total measure of a compact Heisenberg manifold $\{\Gamma_{\bm{r}}\backslash H_n,dist_{A_k}\}$ is
		\begin{align*}
			meas(\Gamma_{\bm{r}}\backslash H_n,A)&:=\left|\int_{\Gamma_{\bm{r}}\backslash H_n}v(A_k)\right|\\
			&=\left|\int_{\Gamma_{\bm{r}}\backslash H_n}\min\{|\rho_{A_k}^{-1}|,\delta(A_k)^{-1}\}\det(\tilde{A}_k)X_1^{\ast}\wedge\cdots\wedge Z^{\ast}\right|\\
			&=\min\{|\rho_{A_k}|^{-1},\delta(A_k)\}|\det(\tilde{A}_k)|^{-1}\prod_{i=1}^nr_i.
	\end{align*}
		Hence,
		if the total measure converges to zero,
		then one of the following two cases holds.
		\begin{itemize}
			\item[(a)]$\min\{|\rho_{A_k}|^{-1},\delta(A_k)^{-1}\}\to 0$,~or
			\item[(b)]$|\det(\tilde{A}_k)|^{-1}\to 0$.
		\end{itemize}

		In the case (a),
by using Lemma \ref{lemfundamental},
\ref{lemvertical} and \ref{lemfiberdiam},
we have
\begin{align*}
	diam(F_b)&=\widetilde{dist}_{A_k}\left(e,\exp\left(\frac{1}{2}Z\right)\right)\\
&=\min\left\{\left|\frac{1}{2\rho_{A_k}}\right|,\frac{2}{d_n(A_k)}\sqrt{\frac{\pi d_n(A_k)}{2}-\pi^2\rho_{A_k}^2}\right\}\\
&\leq \min\left\{\left|\frac{1}{2\rho_{A_k}}\right|,\sqrt{\frac{2\pi}{d_n(A_k)}}\right\}\\
&\leq \min\left\{\left|\frac{1}{2\rho_{A_k}}\right|,\frac{2\sqrt{n\pi}}{\delta(A_k)}\right\}\to 0~~(k\to\infty).
\end{align*}

In the case (b),
again by using Lemma \ref{lemfundamental},
\ref{lemvertical} and \ref{lemfiberdiam},
we have
\begin{align*}
	diam(F_b)&=\min\left\{\left|\frac{1}{2\rho_{A_k}}\right|,\frac{2}{d_n(A_k)}\sqrt{\frac{\pi d_n(A_k)}{2}-\pi^2\rho_{A_k}^2}\right\}\\
	&\leq\sqrt{\frac{2\pi}{d_n(A_k)}}\\
	&\leq \sqrt{\frac{2\pi}{\sqrt[n]{|\det(\tilde{A}_k)|}}}\to 0 ~~(k\to\infty).
\end{align*}

In both cases,
the diameter of the fiber $F_b$ converges to zero.
This concludes the proposition.
\end{proof}

	Trivially Lemma \ref{lem9-2},
	Proposition \ref{prop9-1} and \ref{prop9-2} show the main theorem.

	\appendix
	
	\section{Appendix}\label{secapp}
	In the appendix,
	we consider the systolic inequality on sub-Riemannian compact Heisenberg manifolds.
	Namely we show that the systolic inequality holds when the Popp's volume is applied,
	and that the inequality fails when the minimal Popp's volume is applied.

	First we recall the definition of systole.
	\begin{definition}
		The ($1$-)systole of a length space $(X,dist)$ is defined by
		$$sys(X,dist):=\inf\{length(c)\mid [c]\in H_1(X,\Z)\setminus\{0\}\}.$$
	\end{definition}

	\begin{example}
		For a torus with a flat Riemannian metric $(\Z^n\backslash \R^n,g)$,
		denote by $dist_g$ the associated distance function and $vol_g$ the Riemannian volume form.
		The systole is written by
		$$sys(\Z^n\backslash \R^n,dist_g)=\min\{\widetilde{dist_g}(0,z)\mid z\in\Z^n\setminus\{0\}\}.$$
	The Minkowski convex body theorem asserts that the systole has an upper bound
	\begin{equation}\label{sysineq}
		sys(\Z^n\backslash\R^n,dist_g)\leq \tilde{C}_n\left|\int_{\Z^n\backslash \R^n}vol_g\right|^{\frac{1}{n}},
	\end{equation}
	where $\tilde{C}_n$ is the constant dependent only on the dimension $n$.
	\end{example}

	The inequality (\ref{sysineq}) is called the systolic inequality.
	It is generalized to a non-flat $2$-dimensional torus (by Loewner unpublished,
	please see \cite{pu}),
	$2$-dimensional projective space \cite{pu},
	closed surfaces \cite{acc,bla},
	and essential manifolds \cite{gro}.

Our aim is to generalize these results to the sub-Riemannian setting with the Popp's volume.
For a matrix $A=\begin{pmatrix}
	\tilde{A} & 0\\
	0 & 0
\end{pmatrix}\in\mathcal{A}\cap\{\rho_A=0\}$,
we will denote by $dist_A$ the associated sub-Riemannian distance on $\Gamma_{\bm{r}}\backslash H_n$ and $v(\mv_0,A)$ the Popp's volume form.

\begin{theorem}\label{thmapp1}

	There is a positive constant $C_n$ dependent only on $n$ such that for every $2n+1$-dimensional compact Heisenberg manifold $(\Gamma_{\bm{r}}\backslash H_n,A)$,
	$$sys(\Gamma_{\bm{r}}\backslash H_n,dist_A)\leq C_n\left|\int_{\Gamma_{\bm{r}}\backslash H_n}v(\mv_0,A)\right|^{\frac{1}{2n+2}}.$$
\end{theorem}

\begin{remark}
The exponent is equal to the inverse of the Hausdorff dimension $\frac{1}{2n+2}$.
It is open whether this coincidence holds for every $2$-step nilmanifolds.
\end{remark}

\begin{proof}
	From its definition,
	the systole of a compact Heisenberg manifold\\
	$(\Gamma_{\bm{r}}\backslash H_n,dist_A)$ is given by
	$$sys(\Gamma_{\bm{r}}\backslash H_n,dist_A)=\min\left\{\widetilde{dist}_A(e,\gamma)\mid \gamma\in\Gamma_{\bm{r}}\setminus\{e\right\}.$$
	By Lemma \ref{lemhorizontal} and \ref{lemvertical},
	the systole is the minimum of $s_1(\bm{r},A)$ and $s_2(A)$,
	where
		\begin{equation}\label{eqs1}
			s_1(\bm{r},A):=\min\{\|X\|_A\mid X\in P_0(\Gamma_{\bm{r}})\subset\mv_0\},\end{equation}
		and
		\begin{equation}\label{eqs2}
			s_2(A):=\min\{dist_A(e,pZ)\mid p\in\Z\}=dist_A(e,\exp(Z))=2\sqrt{\frac{\pi}{d_n(A)}}.
		\end{equation}
	
	Notice that $s_1(\bm{r},A)$ is the systole of the base flat torus $(P_0(\Gamma_{\bm{r}})\backslash \mv_0,g_{\tilde{A}})$ and $s_2(A)$ is that of the circle fiber.
	The total measure of the base flat torus is
	$$\left|\int_{P_0(\Gamma_{\bm{r}})\backslash\mv_0}vol_{g_{\tilde{A}}}\right|=\prod_{i_1}^nr_i|\det(\tilde{A})|^{-1},$$
and the systolic inequality on the torus is
\begin{equation}\label{ineqsystorus}
	s_1(\bm{r},A)\leq \tilde{C}_{2n}\left(\prod r_i|\det(\tilde{A})|^{-1}\right)^{\frac{1}{2n}}.
\end{equation}
Combining (\ref{eqs2}) and (\ref{ineqsystorus}) and Lemma \ref{lemfundamental},
we have
\begin{align*}
	sys(\Gamma_{\bm{r}}\backslash H_n,dist_A)^{2n+2}&\leq s_1(\bm{r},A)^{2n}s_2(A)^2\\
	&\leq 4\pi \tilde{C}_{2n}^{2n}\prod r_i|\det(\tilde{A})|^{-1}d_n(A)^{-1}\\
&\leq \frac{2\sqrt{2}\pi \tilde{C}_{2n}^{2n}}{\sqrt{n}}\prod r_i|\det(\tilde{A})|^{-1}\delta(A)^{-1}\\
&=\frac{2\sqrt{2}\pi \tilde{C}_{2n}^{2n}}{\sqrt{n}}\left|\int_{\Gamma_{\bm{r}}\backslash H_n}v(\mv_0,A)\right|.
\end{align*}
	
Hence we can choose $C_n=\left(\frac{2\sqrt{2}\pi \tilde{C}_{2n}^{2n}}{\sqrt{n}}\right)^{\frac{1}{2n+2}}$.
\end{proof}

In the above argument,
the optimality of the constant $C_n$ is not discussed.
The optimal constant is given in the $3$-dimensional case by using the optimal constant of a $2$-dimensional torus $\tilde{C}_2=\sqrt{\frac{2}{\sqrt{3}}}$.

In the following theorem,
notice that the diffeomorphism class of $3$-dimensional compact Heisenberg manifold is classified by a natural number $r\in D_1=\N$.

\begin{theorem}\label{thmapp2}
	For $r\in D_1$,
	there is a constant $C_r$ such thet

	\begin{itemize}
		\item[(1)]If $r\leq 4\pi\tilde{C}_2^2$,
		 then $sys(\Gamma_r\backslash H_1)=s_1(r,A)$ and 
		 $$sys(\Gamma_r\backslash H_1,dist_A)\leq C_r\left|\int_{\Gamma_r\backslash H_1}v(\mv_0,A)\right|^{\frac{1}{4}}.$$
		 Moreover the equality holds if and only if the quotient flat metric on the torus satisfies the equality condition of the systolic inequality (\ref{ineqsystorus}).
		 \item[(2)]If $r\geq 4\pi\tilde{C_2}^2$,
		 then
		 $$sys(\Gamma_r\backslash H_1,dist_A)\leq C_r\left|\int_{\Gamma_r\backslash H_1}v(\mv_0,A)\right|^{\frac{1}{4}}.$$
		 Moreover the equality holds if and only if $s_1(r,A)\geq s_2(A)$.

	\end{itemize}
\end{theorem}

\begin{proof}
	By the systolic inequality (\ref{ineqsystorus}),
	we have
	\begin{equation}
		\frac{s_1(r,A)}{\left|\int_{\Gamma_r\backslash H_1}v(\mv_0,A)\right|^{\frac{1}{4}}}=\frac{s_1(r,A)}{r^{\frac{1}{4}}|\det(\tilde{A})|^{-\frac{1}{2}}}=\frac{s_1(r,A)}{(r|\det(\tilde{A})|^{-1})^{\frac{1}{2}}}\cdot \sqrt[4]{r}\leq\tilde{C}_2\sqrt[4]{r}.\label{ineqfracs1}
	\end{equation}
	On the other hand,
	by Lemma \ref{lemfundamental} and \ref{lemvertical},
	we have
	\begin{equation}
		\frac{s_2(A)}{\left|\int_{\Gamma_r\backslash H_1}v(\mv_0,A)\right|^{\frac{1}{4}}}=\frac{2\sqrt{\frac{\pi}{d_1(A)}}}{\left(r\left|\det(\tilde{A})\right|^{-1}\right)^{\frac{1}{4}}}=\frac{2\sqrt{\frac{\pi}{d_1(A)}}}{\left(rd_1(A)^{-1}\right)^{\frac{1}{4}}}=\frac{2\sqrt{\pi}}{\sqrt[4]{r}}.\label{eqfracs2}
	\end{equation}

Thus the ratio of the systole and the total measure is bounded by
\begin{equation}
	\frac{sys(\Gamma_r\backslash H_1,dist_A)}{\left|\int_{\Gamma_r\backslash H_1}v(\mv_0,A)\right|}=\frac{\min\{s_1(r,A),s_2(A)\}}{\left|\int_{\Gamma_r\backslash H_1}v(\mv_0,A)\right|}
	\leq\min\left\{\tilde{C}_2\sqrt[4]{r},\frac{2\sqrt{\pi}}{\sqrt[4]{r}}\right\}.
\end{equation}

If $\tilde{C}_2\sqrt[4]{r}\leq \frac{2\sqrt{\pi}}{\sqrt[4]{r}}$,
then the systole of the base flat torus is equal to the ambient Heisenberg manifold,
and the assertion of the case (1) holds with the constant $C_r=\tilde{C}_2\sqrt[4]{r}$.

If $\tilde{C}_2\sqrt[4]{r}\geq \frac{2\sqrt{\pi}}{\sqrt[4]{r}}$,
then the case (2) holds with the constant $C_r=\frac{2\sqrt{\pi}}{\sqrt[4]{r}}$,
and the equality holds if and only if the systole of the circle fiber is equal to the ambient Heisenberg manifold.

\end{proof}

The above theorems are assertions on the Popp's volume.
However it is difficult to show the systolic inequality for the minimal Popp's volume.
The reason is the choice of the exponent of the total measure.
As we can see in the following example,
there is no hope to obtain the systolic inequality with a fixed exponent.

\begin{example}
	Let us consider a compact Heisenberg manifold $\Gamma_1\backslash H_1$ for $1\in D_1$ and the following sub-Riemannian metrics.
Let $\{A_k\}_{k\in\N}$ be a sequence in $\mathcal{A}$ given by
$$A_k=\begin{pmatrix}
	k & 0 & 0\\
0 & k & 0\\
0 & 0 & 1
\end{pmatrix}.$$
A straightforward calculation shows that
$$sys(\Gamma_1\backslash H_1,dist_{A_k})=k^{-1},$$
and
$$\left|\int_{\Gamma_1\backslash H_1}v(A_k)\right|=\frac{1}{\sqrt{2}}k^{-4}.$$
This implies that for any positive number $C>0$,
there is $k\in\N$ such that
$$sys(\Gamma_1\backslash H_1,dist_{A_k})\geq C\left|\int_{\Gamma_1\backslash H_1}v(A_k)\right|^{\frac{1}{3}}.$$

Let $\{B_k\}_{k\in\N}$ be a sequence of matrices in $\mathcal{A}$ given by
$$B_k=\begin{pmatrix}
	k^{-1} & 0 & 0\\
	0 & k^{-1} & 0\\
	0 & 0 & k^{-1}
\end{pmatrix}.$$
Then we have
$$sys(\Gamma_1\backslash H_1,dist_{B_k})=k$$
and
$$\left|\int_{\Gamma_1\backslash H_1}v(B_k)\right|=k^3.$$
This implies that for any positive number $C>0$,
there is $k\in \N$ such that
$$sys(\Gamma_1\backslash H_1,dist_{B_k})\geq C\left|\int_{\Gamma_1\backslash H_1}v(B_k)\right|^{\frac{1}{4}}.$$

The above two examples show that the systolic inequality does not holds with a fixed exponent.
\end{example}

\end{document}